\documentclass{amsart}
\usepackage[all]{xy}
\usepackage{amsmath,amsfonts,amssymb,amsthm,epsfig,amscd,comment,latexsym,psfrag}
\usepackage{nicefrac,xspace,tikz}
\usetikzlibrary{arrows}

\title{On a categorical Boson-Fermion correspondence}
\author{Sabin Cautis}
\email{cautis@math.ubc.ca}
\address{Department of Mathematics\\ University of British Columbia \\ Vancouver, Canada}

\author{Joshua Sussan}
\email{jsussan@mec.cuny.edu}
\address{Department of Mathematics\\ CUNY Medgar Evers \\ New York, United States}

\newtheorem{Theorem}{Theorem}[section]

\newtheorem{prop}[Theorem]{Proposition}
\newtheorem{conj}[Theorem]{Conjecture}
\newtheorem{theorem}[Theorem]{Theorem}
\newtheorem{corollary}[Theorem]{Corollary} 
\newtheorem{lemma}[Theorem]{Lemma}
\newtheorem{remark}[Theorem]{Remark}
 
 \newcommand{\Se}{{\mathbb{S}}}
 \newcommand{\Cl}{{\mathcal{C}\ell}}

\begin{document} 
\setcounter{tocdepth}{1}
\baselineskip 14pt

\renewcommand{\mod}{{\operatorname{-mod}}}

\def\A{{\sf{A}}}
\def\Cliff{{\sf{Cliff}}}
\def\T{{\sf{\Sigma}}}
\def\Cone{{\rm{Cone}}}
\def\v{{\sf v}}
\def\1{{\bf 1}}
\def\V{{\sf V}}
\def\k{\Bbbk}
\def\P{\mathsf{P}}
\def\Q{\mathsf{Q}}
\def\X{\mathsf{X}}
\def\sl{{\mathfrak{sl}}}

\def\Kom{{\sf Kom}}
\def\End{{\rm{End}}}
\def\id{{\rm{id}}}
\def\Id{{\rm{id}}}
\def\la{\langle}
\def\ra{\rangle}
\def\l{\lambda}
\def\h{\mathfrak{h}}
\def\th{{\mathfrak{h}}^t}
\def\cl{{\rm{Cl}}}
\def\tcl{{\widetilde{{\rm{Cl}}}}}
\def\H{\mathcal{H}}
\def\tH{{\mathcal{H}}^t}
\def\adj{{{adj}}} 
\def\tr{\text{tr}}
\def\Z{\mathbb Z}
\def\R{\mathbb R}
\def\N{\mathbb N} 
\def\o{\otimes}
\def\lra{\longrightarrow}
\def\Hom{\mathrm{Hom}}
\def\HOM{\mathrm{HOM}}
\def\RHom{\mathrm{RHom}}
\def\mc{\mathcal}
\def\mf{\mathfrak} 
\def\Ext{\mathrm{Ext}}
\def\Ind{\mathrm{Ind}}
\def\Res{\mathrm{Res}}
\def\soc{\mathrm{soc}}
\def\Kom{{\mathsf{Kom}}}
\def\Seq{\mathrm{Seq}}
\def\triv{{\rm{triv}}}
\def\C{{\sf{C}}}

\def\shuffle{\,\raise 1pt\hbox{$\scriptscriptstyle\cup{\mskip
               -4mu}\cup$}\,}
\newcommand{\define}{\stackrel{\mbox{\scriptsize{def}}}{=}}

\def\drawing#1{\begin{center}\epsfig{file=#1}\end{center}}

 \def\yesnocases#1#2#3#4{\left\{
\begin{array}{ll} #1 & #2 \\ #3 & #4
\end{array} \right. }

\begin{abstract} 
We propose a categorical version of the Boson-Fermion correspondence and its twisted version. One can view it as a relative of the Frenkel-Kac-Segal construction of quantum affine algebras. 
\end{abstract}

\maketitle
\tableofcontents 

\section{Introduction}

The Boson-Fermion correspondence can be interpreted as a relationship between a particular Heisenberg algebra $\h$ and an associated Clifford algebra $\cl$.  Its mathematical formulation was introduced in \cite{F}. It gives a recipe for constructing an action of $\cl$ given an action of $\h$ (the precise construction is given in section \ref{sec:corr}, Theorem \ref{thm:A}). 

In this paper we propose a categorical version of this result. The role of $\h$ is now played by the monoidal category $\H$ introduced by Khovanov \cite{Kh1}. This category contains not only the generators of $\h$ but also a set of natural transformations between them. The categorical action of $\cl$ is now given by certain complexes in $\H$ (see equations (\ref{eq:cpxC1}) and (\ref{eq:cpxC2}) as well as Conjecture \ref{conj:1}). The extra structure of natural transformations in $\H$ is used to define these complexes. 

There is a close analogy between the Boson-Fermion correspondence and the Frenkel-Kac-Segal (FKS) construction \cite{FK,Seg}. In both cases one starts with a Heisenberg algebra and then defines certain vertex operators. In the first case these operators induce an action of the Clifford algebra whereas in the second case they induce an action of a (quantum) affine algebra. 

In \cite{CL2} we categorified the FKS construction. More precisely, starting with the Heisenberg category studied in \cite{CL1} we defined certain complexes (categorical vertex operators) and showed that these induce categorical actions of quantum affine algebras. The conjectures in this paper are analogous. Instead of the Heisenberg from \cite{CL1} we use $\H$ from \cite{Kh1}. Then we write down certain complexes and conjecture they should give an action of a Clifford algebra. Although this action is only conjectural, the proof should follow along the same lines as that from \cite{CL2}. 

In the second part of the paper we discuss the twisted version of the Boson-Fermion correspondence. The category $\H$ is replaced by $\tH$. Since $\tH$ was not previously studied we discuss it in a little more detail (sections \ref{sec:tH}, \ref{sec:sergeev}, \ref{sec:fock}). This story is the twisted version of \cite{Kh1}. We then define again certain complexes and conjecture that they induce an action of a (cover of the) twisted Clifford algebra $\tcl^t$. This is a categorical analogue of the twisted Boson-Fermion correspondence. 

In \cite{CLS} we studied certain complexes $\T_i$ living in the Heisenberg category from \cite{CL1} which turn out to satisfy the braid relations. These complexes have an analogue $\T$ which makes sense in the context of the category $\H$. In section \ref{sec:projectors} we define $\T$ and conjecture that it is an idempotent. This discussion is not directly related to the Boson-Fermion correspondence and section \ref{sec:projectors} can be read independently of the others. 

A closely related article was posted by Frenkel, Penkov and Serganova \cite{FPS}. 
They study complexes of functors on a carefully chosen subcategory of representations of $\mathfrak{sl}_{\infty}$.
In particular, the complexes $X_a$ and $X_a^*$ from \cite{FPS} correspond to $\Psi_i$ and $\Psi_i^*$ in this paper. One difference between their setup and ours is that they do not assume biadjointness in their Heisenberg generators (functors such as ${\mathcal E}_n$ and ${\mathcal E}_n^*$ in \cite{FPS}). This means that they have adjunction maps in one direction but not the other. Subsequently one can define complexes (categorical vertex operators) only in one direction. 

The lack of biadjointness in \cite{FPS} has the advantage that their categories are less semi-simple. One would expect, or at least hope, that our conjectures (in particular Conjecture \ref{conj:1}) should have a direct analogue in their work. 

\noindent {\bf Acknowledgments:}
The authors would like to thank Mikhail Khovanov and Yin Tian for helpful discussions and Alistair Savage for pointing out a missing relation in Proposition \ref{prop:3} in an earlier version of the paper. S.C. was supported by NSF grant DMS-1332847 and both authors received support from the Templeton foundation. 

\section{The Boson-Fermion correspondence}

Throughout this paper we will work over an arbitrary field $\k$ of characteristic zero. The Boson-Fermion correspondence is a relationship between two algebras: a Heisenberg algebra $\h$ and a Clifford algebra $\cl$. 

\subsection{Heisenberg algebra $\h$}

In general, one can associate to any $\Z$-lattice a Heisenberg algebra. In this case we are interested in the simplest Heisenberg algebra, namely that associated with $\Z$ where the inner product is $\la 1,1 \ra = 1$. This is an associative unital $\k$-algebra, denoted $\h$, with generators $h_n$ for $n \in \Z \setminus \{0\}$ and relations $[h_m,h_n] = m \delta_{m,-n}$. 

Another way to represent $\h$ is as follows ({\it c.f.} \cite{Kh1}). For $m \ge 0$ define $p^{(m)}$ and $q^{(m)}$ as follows:
$$\sum_{m \in \N} p^{(m)} z^m = \exp \left( \sum_{m \in \N} \frac{h_{-m}}{m} z^m \right) \ \ \text{ and }  \ \ \sum_{m \in \N} q^{(m)} z^m = \exp \left( \sum_{m \in \N} \frac{h_{m}}{m} z^m \right).$$
By convention we take $p^{(m)}=q^{(m)}=0$ if $m < 0$. Note that $p^{(0)} = q^{(0)} = 1$. 

\begin{prop}\label{prop:1}
\label{heisrel1}
The Heisenberg algebra $\h$ is generated by $p^{(m)}, q^{(m)}$, $m \in \N$ with relations
\begin{enumerate}
\item $ p^{(n)} p^{(m)} = p^{(m)} p^{(n)} $ and $ q^{(n)} q^{(m)} = q^{(m)} q^{(n)} $ for any $m,n$,
\item $ q^{(n)} p^{(m)} = \sum_{k \ge 0} p^{(m-k)} q^{(n-k)}$.
\end{enumerate}
\end{prop}
In a similar way one can also define generators $p^{(1^m)}$ and $q^{(1^m)}$ by adding a minus sign in the expressions above, namely 
$$\sum_{m \in \N} (-1)^m p^{(1^m)} z^m = \exp \left( - \sum_{m \in \N} \frac{h_{-m}}{m} z^m \right) \ \ \text{ and }  \ \  \sum_{m \in \N} (-1)^m q^{(1^m)} z^m = \exp \left( \sum_{m \in \N} - \frac{h_{m}}{m} z^m \right).$$

\begin{prop}\label{prop:2}
\label{heisrel2}
The following relations hold inside $\h$:
\begin{enumerate}
\item all $p$'s commute among each other, for example $p^{(m)} p^{(1^n)} = p^{(1^n)} p^{(m)}$,\item all $q$'s commute among each other, 
\item $q^{(n)} p^{(1^m)} = p^{(1^m)} q^{(n)} + p^{(1^{m-1})} q^{(n-1)}$ and $q^{(1^n)} p^{(m)} = p^{(m)} q^{(1^n)} + p^{(m-1)} q^{(1^{n-1})}$. 
\end{enumerate}
\end{prop}

The relations in the Propositions above are either clear (such as the first relation in Prop. \ref{prop:1}) or appear in \cite{Kh1} (either explicitly such as the last relation in Prop. \ref{prop:2} \cite[Prop. 1]{Kh1} or implicitly such as the last relation in Prop. \ref{prop:1}).

The algebra $\h$ has an irreducible representation $V_{Fock} := {\rm Ind}_{\h^+}^{\h}({\triv})$ where $\h^+ \subset \h$ is the subalgebra generated by $q^{(m)}$'s and $\triv$ is its one dimensional representation where $q^{(m)}$ acts by zero if $m > 0$. Note that $V_{Fock}$ is spanned by products of the form $p^{(m_1)} p^{(m_2)} \dots p^{(m_k)}$.  

\subsection{The correspondence}\label{sec:corr}

The Clifford algebra $\cl$ has generators $\psi_i$ and $\psi_i^*$ ($i \in \Z$) and relations 
\begin{equation}\label{eq:cl1}
\psi_i \psi_j + \psi_j \psi_i = 0, \ \ \psi_i^* \psi_j^* + \psi_j^* \psi_i^* = 0 \ \ \text{ and } \ \ \psi_i \psi_j^* + \psi_j^* \psi_i = \delta_{i,j}.
\end{equation}
Consider the vector space $V := \k[e^{\pm \alpha}] \otimes_\k V_{Fock}$.

\begin{theorem}[Boson-Fermion correspondence]\label{thm:A}
One can define an action of $\cl$ on $V$ by 
$$\psi_i(e^{n \alpha} \otimes v) := e^{(n+1)\alpha} \otimes C^-_{i+n}(v) \ \ \text{ and } \ \ \psi_i^*(e^{(n+1) \alpha} \otimes v) := e^{n \alpha} \otimes C^+_{i+n}(v)$$
where $C_i^-$ and $C_i^+$ are defined by 
\begin{align}
\label{eq:C1} C_i^- &= \begin{cases}
\sum_{k \ge 0} (-1)^{k} p^{(k)} q^{(1^{i+k})} & \text{ if } i \ge 0 \\
\sum_{k \ge 0} (-1)^{i+k} p^{(-i+k)} q^{(1^k)} & \text{ if } i \le 0
\end{cases} \\
\label{eq:C2} C_i^+ &= \begin{cases}
\sum_{k \ge 0} (-1)^{k} p^{(1^{i+k})} q^{(k)} & \text{ if } i \ge 0 \\
\sum_{k \ge 0} (-1)^{i+k} p^{(1^k)} q^{(-i+k)} & \text{ if } i \le 0.
\end{cases} 
\end{align}
\end{theorem}
\begin{remark} Notice that the sums in (\ref{eq:C1}) and (\ref{eq:C2}) used to define $C_i^{\pm}$, although infinite, are in fact finite when applied to $V_{Fock}$ (and hence also to $V$). This is because for any $v \in V_{Fock}$ we have $q^{(n)}(v)=0$ for $n \gg 0$. 
\end{remark}
\begin{proof}
See ~\cite{F} or ~\cite[section 2]{An}.
\end{proof}

\section{The categorical Boson-Fermion correspondence}

\subsection{The Heisenberg category $\H$}\label{sec:H}

In \cite{Kh1} Khovanov introduced a categorical framework for the Heisenberg algebra $\h$. This framework consists of an idempotent complete monoidal category $\H$ whose definition we now sketch (see \cite{Kh1} for more details). 

The category $\H$ is generated by objects $\P$ and $\Q$. These can be denoted by an upward pointing strand and a downward pointing strand. Monoidal composition of such objects is then given by sideways concatenation of diagrams. The space of morphisms between products of $\P$'s and $\Q$'s is a $\k$-algebra described by certain string diagrams with relations. By convention, composition of morphisms is done vertically from the bottom and going up.

The morphisms are generated by crossings, caps and cups as shown below
\begin{equation}\label{eq:maps} 
\begin{tikzpicture}[scale=.75][>=stealth]
\draw [->](0,0) -- (1,1) [thick];
\draw [->](1,0) -- (0,1) [thick];
\draw (3,.5) arc (180:360:.5)[->] [thick];
\draw (6.5,.25) arc (0:180:.5) [->][thick];
\draw (9,.5) arc (180:360:.5)[<-] [thick];
\draw (12.5,.25) arc (0:180:.5) [<-][thick];
\end{tikzpicture} 
\end{equation}
Thus, for instance, the left crossing is a map in $\End(\P\P)$ while the right cap is a map $\P \Q \rightarrow \id$. We will abuse notation slightly and denote all these maps $\adj$ since they are all adjunction maps between $\P$'s and $\Q$'s. 

These morphisms satisfy the following relations 
\begin{equation}\label{eq:rel1}
\begin{tikzpicture}[>=stealth]
\draw [shift={+(7,0)}](0,0) .. controls (1,1) .. (0,2)[->][thick] ;
\draw [shift={+(7,0)}](1,0) .. controls (0,1) .. (1,2)[->] [thick];
\draw [shift={+(7,0)}](1.5,1) node {=};
\draw [shift={+(7,0)}](2,0) --(2,2)[->][thick] ;
\draw [shift={+(7,0)}](3,0) -- (3,2)[->][thick] ;

\draw (0,0) -- (2,2)[->][thick];
\draw (2,0) -- (0,2)[->][thick];
\draw (1,0) .. controls (0,1) .. (1,2)[->][thick];
\draw (2.5,1) node {=};
\draw (3,0) -- (5,2)[->][thick];
\draw (5,0) -- (3,2)[->][thick];
\draw (4,0) .. controls (5,1) .. (4,2)[->][thick];
\end{tikzpicture}
\end{equation}

\begin{equation}\label{eq:rel2}
\begin{tikzpicture}[>=stealth]
\draw (0,0) .. controls (1,1) .. (0,2)[<-][thick];
\draw (1,0) .. controls (0,1) .. (1,2)[->] [thick];
\draw (1.5,1) node {=};
\draw (2,0) --(2,2)[<-][thick];
\draw (3,0) -- (3,2)[->][thick];

\draw (3.8,1) node{$-$};

\draw (4,1.75) arc (180:360:.5) [thick];
\draw (4,2) -- (4,1.75) [thick];
\draw (5,2) -- (5,1.75) [thick][<-];
\draw (5,.25) arc (0:180:.5) [thick];
\draw (5,0) -- (5,.25) [thick];
\draw (4,0) -- (4,.25) [thick][<-];

\draw [shift={+(7,0)}](0,0) .. controls (1,1) .. (0,2)[->][thick];
\draw [shift={+(7,0)}](1,0) .. controls (0,1) .. (1,2)[<-] [thick];
\draw [shift={+(7,0)}](1.5,1) node {=};
\draw [shift={+(7,0)}](2,0) --(2,2)[->][thick];
\draw [shift={+(7,0)}](3,0) -- (3,2)[<-][thick];
\end{tikzpicture}
\end{equation}

\begin{equation}\label{eq:rel3}
\begin{tikzpicture}[>=stealth]
\draw [shift={+(0,0)}](0,0) arc (180:360:0.5cm) [thick];
\draw [shift={+(0,0)}][->](1,0) arc (0:180:0.5cm) [thick];
\draw [shift={+(0,0)}](1.75,0) node{$= 1.$};

\draw  [shift={+(5,0)}](0,0) .. controls (0,.5) and (.7,.5) .. (.9,0) [thick];
\draw  [shift={+(5,0)}](0,0) .. controls (0,-.5) and (.7,-.5) .. (.9,0) [thick];
\draw  [shift={+(5,0)}](1,-1) .. controls (1,-.5) .. (.9,0) [thick];
\draw  [shift={+(5,0)}](.9,0) .. controls (1,.5) .. (1,1) [->] [thick];
\draw  [shift={+(5,0)}](1.5,0) node {$=$};
\draw  [shift={+(5,0)}](2,0) node {$0.$};
\end{tikzpicture}
\end{equation}

Moreover, two morphisms which differ by planar isotopies are equal. Relation (\ref{eq:rel1}) implies that there is a map $\k[S_n] \rightarrow \End(\P^n)$. Since $\H$ is idempotent complete this means that we also get objects $\P^{(\l)}$, for any partition $\l \vdash n$, associated with the corresponding minimal idempotent $e_\l \in \k[S_n]$. Likewise we also have $\Q^{(\l)}$ for any $\l \vdash n$. We will denote by $(m)$ and $(1^m)$ the unique one-part and $m$-part partitions of $m$.  The object $ \P^{\lambda} $ will be denoted pictorially by a box labeled by the partition $\lambda$ with upward pointing arrows coming out from the top.  The object $\Q^{\lambda} $ will be denoted similarly except that downward pointing arrows will be coming out from the top.

\begin{theorem}[\cite{Kh1}]\label{thm:1}
Inside $\H$ we have the following relations 
\begin{enumerate}
\item $\P^{(\l)} \P^{(\mu)} \cong \P^{(\mu)} \P^{(\l)}$ for any partitions $\l,\mu$,
\item $\Q^{(\l)} \Q^{(\mu)} \cong \Q^{(\mu)} \Q^{(\l)}$ for any partitions $\l,\mu$,
\item $\Q^{(n)} \P^{(m)} \cong \bigoplus_{k \ge 0} \P^{(m-k)} \Q^{(n-k)}$ and $\Q^{(1^n)} \P^{(1^m)} \cong \bigoplus_{k \ge 0} \P^{(1^{m-k})} \Q^{(1^{n-k})}$,
\item $\Q^{(n)} \P^{(1^m)} \cong \P^{(1^m)} \Q^{(n)} \oplus \P^{(1^{m-1})} \Q^{(n-1)}$ and $\Q^{(1^n)} \P^{(m)} \cong \P^{(m)} \Q^{(1^n)} \oplus \P^{({m-1})} \Q^{(1^{n-1})}$. 
\end{enumerate}
\end{theorem}
\begin{proof}
The first and second relations are clear because all $\P$'s commute among themselves and likewise for $\Q$'s. Alternatively one applies the same proof as for the last two relations in \cite[Prop. 1]{Kh1}. The fourth relation appears in \cite[Prop. 1]{Kh1}. Although the third relation does not explicitly appear in \cite{Kh1} it can be easily derived in the same way as the fourth relation (or alternatively as in the proof of \cite[Prop. 2]{CL1} where one uses the same calculation but without having any dots). 
\end{proof}

Thus, at the level of Grothendieck groups we have a map $\h \rightarrow K_0(\H)$. This map is known to be injective but it is not known if it is surjective. 

In this categorical context there is also an analogue of the Fock space $V_{Fock}$. One such analogue is given by the direct sum $\oplus_{n \ge 0} \k[S_n]\mod$, as discussed in \cite{Kh1}. Though it may not be unique, we will fix from now on such a categorical representation $\V_{Fock}$. In practice the only thing we will use of $\V_{Fock}$ is that on any object, applying $\Q$ sufficiently many times gives zero. 

\subsection{Complexes}

We will need to work with the homotopy category $\Kom(\H)$ of $\H$. Here the objects are unbounded complexes of objects in $\H$ while morphisms are maps of complexes. Two objects are deemed equivalent if there is a homotopy equivalence between them. This gives $\Kom(\H)$ the structure of a triangulated category. 

We can now define the analogues of $C_i^\pm$ from (\ref{eq:C1}) and (\ref{eq:C2}) as follows
\begin{equation}
\label{eq:cpxC1} \C^-_i := 
\begin{cases}
\left( \dots \rightarrow \P^{(k)} \Q^{(1^{i+k})} \rightarrow \dots \rightarrow \P \Q^{(1^{i+1})} \rightarrow \Q^{(1^i)} \right) & \text{ if } i \ge 0 \\
\left( \dots \rightarrow \P^{(-i+k)} \Q^{(1^k)} \rightarrow \dots \rightarrow \P^{(-i+1)} \Q \rightarrow \P^{(-i)} \right) [-i]  & \text{ if } i \le 0 
\end{cases}
\end{equation}
where the right most term is in cohomological degree zero and $[1]$ indicates a negative cohomological shift by $1$. The differential here is given by the composition 
$$\P^{(k)} \Q^{(1^{i+k})} \rightarrow \P^{(k-1)} \P \Q \Q^{(1^{i+k-1})} \xrightarrow{I \adj I} \P^{(k-1)} \Q^{(1^{i+k-1})}$$ 
where the first map consists of the inclusions 
$$\P^{(k)} \hookrightarrow \P^{(k-1)} \P \cong \P^{(k)} \oplus \P^{(k-1,1)} \ \ \text{ and } \ \ \Q^{(1^{i+k})} \hookrightarrow \Q \Q^{(1^{i+k-1})} \cong \Q^{(1^{i+k})} \oplus \Q^{(2,1^{i+k-2})}$$ 
while the second map uses the adjunction $\adj: \P \Q \rightarrow \id$. Diagrammatically this is given by 
\begin{equation} \label{eq:diff0}
\begin{tikzpicture}[>=stealth]
\draw (9.75,0) rectangle (11.25,.5);
\draw (10.5,.25) node {$(k)$};
\draw (9.75,1.5) rectangle (11.25,2);
\draw (10.5,1.75) node {$(k-1)$};
\draw (10.5,.5) -- (10.5,1.5) [->][thick];
\draw (12.5,.5) -- (12.5,1.5) [<-][thick];
\draw (11.75,0) rectangle (13.25,.5);
\draw (12.5,.25) node {$(1^{i+k})$};
\draw (11.5,1.5) rectangle (13.5,2);
\draw (12.5,1.75) node {$(1^{i+k-1})$};
\draw (11,.5) arc (180:0:0.5cm)[->] [thick];
\end{tikzpicture}
\end{equation}
Likewise, we have
\begin{equation}
\label{eq:cpxC2} \C^+_i := 
\begin{cases}
\left( \P^{(1^{i})} \rightarrow \P^{(1^{i+1})} \Q \rightarrow \dots \rightarrow \P^{(1^{i+k})} \Q^{(k)} \rightarrow \dots \right) & \text{ if } i \ge 0 \\
\left( \Q^{(-i)} \rightarrow \P \Q^{(-i+1)} \rightarrow \dots \rightarrow \P^{(1^k)} \Q^{(-i+k)} \rightarrow \dots \right) [i] & \text{ if } i \le 0
\end{cases}
\end{equation}
where the left most term is in cohomological degree zero and the differential is given by 
\begin{equation} \label{eq:diff1}
\begin{tikzpicture}[>=stealth]
\draw (9.5,0) rectangle (11.5,.5);
\draw (10.5,1.75) node {$(i+k)$};
\draw (9.75,1.5) rectangle (11.25,2);
\draw (10.5,.25) node {$(i+k-1)$};
\draw (10.5,.5) -- (10.5,1.5) [->][thick];
\draw (12.5,.5) -- (12.5,1.5) [<-][thick];
\draw (11.75,0) rectangle (13.25,.5);
\draw (12.5,1.75) node {$(1^k)$};
\draw (11.75,1.5) rectangle (13.25,2);
\draw (12.5,.25) node {$(1^{k-1})$};
\draw (11,1.5) arc (-180:0:0.5cm)[<-] [thick];
\end{tikzpicture}
\end{equation}

\subsection{A categorical action of $\cl$}

The analogue of $V$ is the disjoint union $\V := \amalg_{n \in \Z} \V_{Fock}$. Abusing notation one could also write $\V := \k[e^{\pm \alpha}] \otimes_\k \V_{Fock}$. 

\begin{conj}[Categorical Boson-Fermion correspondence]\label{conj:1}\ \\
On $\Kom(\V)$ define functors
$$\Psi_i(e^{n \alpha} \otimes \v) := e^{(n+1) \alpha} \otimes \C_{i+n}^-(\v) \ \ \text{ and } \ \ \Psi_i^*(e^{(n+1) \alpha} \otimes \v) := e^{n \alpha} \otimes \C_{i+n}^+(\v).$$
These functors satisfy the following relations in $\Kom(\H)$
\begin{enumerate}
\item \label{R2} $\Psi_i^2 \cong 0 \cong (\Psi_i^*)^2$ while $\Psi_i \Psi_j \cong \begin{cases} \Psi_j \Psi_i [-1] &\text{ if } i < j \\ \Psi_j \Psi_i [1] & \text{ if } i > j \end{cases}$ and $\Psi^*_i \Psi^*_j \cong \begin{cases} \Psi^*_j \Psi^*_i [-1] &\text{ if } i < j \\ \Psi^*_j \Psi^*_i [1] & \text{ if } i > j \end{cases}$ 
\item \label{R3} $\Psi_i \Psi_j^* \cong \begin{cases} \Psi_j^* \Psi_i [1] & \text{ if } i < j \\ \Psi_j^* \Psi_i [-1] & \text{ if } i > j \end{cases}$
\item \label{R4} there exists a distinguished triangle $\Psi_i \Psi_i^* \rightarrow \id \rightarrow \Psi_i^* \Psi_i$.
\end{enumerate}
\end{conj}
The relations above are categorical analogues of the relations (\ref{eq:cl1}) in section ~\ref{sec:corr}. Note that it is easy to see that $\Psi_i^*$ is both left and right adjoint to $\Psi_i$ because the left and right adjoints of $\P^{(n)}$ and $\P^{(1^n)}$ are $\Q^{(n)}$ and $\Q^{(1^n)}$ respectively. 

\subsection{Examples: untwisted case}\label{sec:exuntwisted}

The relations in Conjecture \ref{conj:1} are based on various calculations. We illustrate this with a couple of examples. The main tool we use is the following cancellation Lemma which allows one to repeatedly simplify a complex in the homotopy category $\Kom(\H)$. This result is a slight generalization of a lemma of Bar-Natan (see \cite[Lem. 6.1]{CL2} for a proof). 

\begin{lemma}\label{lem:cancel}
Let $ X, Y, Z, W, U, V$ be six objects in an additive category and consider a complex
\begin{equation}\label{eq:A}
\dots \rightarrow U \xrightarrow{u} X \oplus Y \xrightarrow{f} Z \oplus W \xrightarrow{v} V \rightarrow \dots
\end{equation}
where $f = \left( \begin{matrix} A & B \\ C & D \end{matrix} \right)$ and $u,v$ are arbitrary morphisms. If $D: Y \rightarrow W$ is an isomorphism, then (\ref{eq:A}) is homotopic to a complex 
$$\dots \rightarrow U \xrightarrow{u} X \xrightarrow{A-BD^{-1}C} Z \xrightarrow{v|_Z} V \rightarrow \dots$$
\end{lemma}

\subsubsection{Example 1}\label{sec:ex1} Suppose that $\Psi_i$ is of the form 
\begin{equation}\label{ex:psi}
\Psi_i = [\dots \rightarrow \P^{(2)} \Q^{(1^2)} \rightarrow \P \Q \rightarrow \id].
\end{equation}
Moreover, to simplify the example we assume $\Q^n = 0$ for $n > 2$. Then 
\begin{align}
\nonumber \Psi_i \Psi_i 
&\cong [\P \Q^{(1^2)} \rightarrow \Q][\P^{(2)} \Q^{(1^2)} \rightarrow \P \Q \rightarrow \id] \\
\nonumber &\cong \left[\P \Q^{(1^2)} \P^{(2)} \Q^{(1^2)} \rightarrow \P \Q^{(1^2)} \P \Q \oplus \Q \P^{(2)} \Q^{(1^2)} \rightarrow \P \Q^{(1^2)} \oplus \Q \P \Q \rightarrow \Q \right] \\
\label{calc:0} &\cong \left[ 0 \rightarrow \P \Q \Q \oplus \P \Q^{(1^2)} \rightarrow \P \Q^{(1^2)} \oplus \Q \oplus \P \Q \Q \rightarrow \Q \right]
\end{align}
where we used that $\Q^{(n)} = \Q^{(1^n)} = 0$ for $n > 2$ to simplify the second line. 

Claim: the maps in (\ref{calc:0}) are isomophisms between the pairs of $\Q, \P\Q\Q$ and $\P\Q^{(1^2)}$. For example, the map between the two $\Q$'s is given by the composition 
$$\Q \xrightarrow{\adj I} \Q \P \Q \xrightarrow{I \adj} \Q$$
which is the identity map by the basic properties of adjunction. Subsequently the complex in (\ref{calc:0}) is contractible and hence $\Psi_i \Psi_i = 0$ (proving relation (\ref{R2}) from Conjecture \ref{conj:1} holds). 

\subsubsection{Example 2}\label{sec:ex2} Suppose we are in the same situation as Example 1 so that $\Psi_i$ is given by (\ref{ex:psi}). We are still assuming $\Q^n = 0$ for $n > 2$. Then 
\begin{align}
\Psi_i^* \Psi_i 
\nonumber &\cong [\id \rightarrow \P \Q \rightarrow \P^{(1^2)} \Q^{(2)}][\P^{(2)} \Q^{(1^2)} \rightarrow \P \Q \rightarrow \id] \\
\nonumber &\cong \left[ \P^{(2)} \Q^{(1^2)} \rightarrow \begin{matrix} \P \Q \oplus \\ \P \Q \P^{(2)} \Q^{(1^2)} \end{matrix} \rightarrow \begin{matrix} \id \oplus \P \Q \P \Q \\ \P^{(1^2)} \Q^{(2)} \P^{(2)} \Q^{(1^2)} \end{matrix} \rightarrow \begin{matrix} \P \Q \oplus \\ \P^{(1^2)} \Q^{(2)} \P \Q \end{matrix} \rightarrow \P^{(1^2)} \Q^{(2)} \right] \\
\label{calc:1} &\cong \left[\P^{(2)} \Q^{(1^2)} \rightarrow \begin{matrix} \P \Q \oplus \\ \P \P \Q^{(1^2)} \end{matrix} \rightarrow \begin{matrix} \id \oplus \P \P \Q \Q \oplus \P \Q \\ \oplus \P^{(1^2)} \Q^{(1^2)} \end{matrix} \rightarrow \begin{matrix} \P \Q \oplus \\ \P^{(1^2)} \Q \Q \end{matrix} \rightarrow \P^{(1^2)} \Q^{(2)} \right] 
\end{align}
where, to obtain the third isomorphism, we used that $\Q^{(n)}=\Q^{(1^n)}=0$ for $n>2$ together with commutation relations such as $\Q^{(2)} \P^{(2)} \cong \P^{(2)} \Q^{(2)} \oplus \P \Q \oplus \id$. Notice that the grading is such that $\id$ is in cohomological degree zero. 

Claim: The map in (\ref{calc:1}) between the two right-most copies of $\P \Q$ is an isomorphism. One can see this because this map is given by the composition
$$\P \Q \xrightarrow{I \adj I} \P \Q \P \Q \xrightarrow{II \adj} \P \Q$$
which is the identity map by the standard properties of adjunction morphisms. Thus, using the cancellation Lemma, one can get rid of these two terms. Similarly, one can show that the piece of the complex $\P^{(1^2)} \Q^{(1^2)} \rightarrow \P^{(1^2)} \Q \Q \rightarrow \P^{(1^2)} \Q^{(2)}$ is homotopic to zero since the first map is the natural inclusion and the second map the natural projection. Thus (\ref{calc:1}) is homotopic to 
$$\left[\P^{(2)} \Q^{(1^2)} \rightarrow \begin{matrix} \P \Q \oplus \\ \P \P \Q^{(1^2)} \end{matrix} \rightarrow \id \oplus \P \P \Q \Q \right].$$
Finally, one can check that the map $\P \P \Q^{(1^2)} \rightarrow \P \P \Q \Q \cong \P \P \Q^{(1^2)} \oplus \P \P \Q^{(2)}$ in the above composition is the natural inclusion. Applying Lemma \ref{lem:cancel} again we are left with 
\begin{equation}\label{calc:2}
\Psi_i^* \Psi_i \cong \left[ \P^{(2)} \Q^{(1^2)} \rightarrow \P \Q \rightarrow \id \oplus \P \P \Q^{(2)} \right]
\end{equation}
where the differentials are the obvious ones ({\it i.e.} given by caps or cups). On the other hand, 
\begin{align}
\nonumber \Psi_i \Psi_i^* &\cong [\P^{(2)} \Q \rightarrow \P][\Q \rightarrow \P \Q^{(2)}] \\
\nonumber &\cong \left[\P^{(2)} \Q \Q \rightarrow \P^{(2)} \Q \P \Q^{(2)} \oplus \P \Q \rightarrow \P \P \Q^{(2)} \right] \\
\nonumber &\cong \left[\P^{(2)} \Q^{(2)} \oplus \P^{(2)} \Q^{(1^2)} \rightarrow \P^{(2)} \Q^{(2)} \oplus \P \Q \rightarrow \P \P \Q^{(2)} \right] \\
\label{calc:3} &\cong \left[\P^{(2)} \Q^{(1^2)} \rightarrow \P \Q \rightarrow \P \P \Q^{(2)} \right]
\end{align}
where we used again Lemma \ref{lem:cancel} to cancel the summand $\P^{(2)} \Q^{(2)}$ in the third line. The differentials here are again the obvious ones. Since the $\P \Q$ in complex (\ref{calc:3}) occurs in degree zero it is easy to check, by comparing with (\ref{calc:2}), that $\Cone(\id \rightarrow \Psi_i^* \Psi_i) \cong \Psi_i \Psi_i^* [1]$ (proving relation \ref{R4} from Conjecture \ref{conj:1} holds).

\section{Aside: complexes and projectors}\label{sec:projectors}

As noted in the introduction, the complexes $\C_i^{\pm}$ from the last section live in $\Kom(\H)$ and are analogous to those studied in \cite{CL2} (which live in a different Heisenberg category $\Kom(\H^\Gamma)$). On the other hand, in \cite{CLS} we studied certain complexes $\T_i$, which again live in $\Kom(\H^\Gamma)$. We showed that these give an action of the braid group. These complexes also have analogues $\T^{\pm} \in \Kom(\H)$. Although $\T^{\pm}$ are not directly related to the Boson-Fermion correspondence this seems like a good place to briefly discuss them. We define
\begin{align}
\label{eq:P1} \T^- &:= \left[ \dots \rightarrow \bigoplus_{\l \vdash d} \P^{(\l)} \Q^{(\l^t)} \rightarrow \bigoplus_{\mu \vdash d-1} \P^{(\mu)} \Q^{(\mu^t)} \rightarrow \dots \rightarrow \P \Q \rightarrow \id \right] \\
\label{eq:P2} \T^+ &:= \left[ \id \rightarrow \P \Q \rightarrow \dots \rightarrow \bigoplus_{\mu \vdash d-1} \P^{(\mu)} \Q^{(\mu^t)} \ \rightarrow \bigoplus_{\l \vdash d} \P^{(\l)} \Q^{(\l^t)} \rightarrow \dots \right]
\end{align}
where $\l^t$ denotes the transpose of $\l$. The differential in (\ref{eq:P1}) is defined by the composition
$$\P^{(\l)} \Q^{(\l^t)} \rightarrow \P^{(\mu)} \P \Q \Q^{(\mu^t)} \rightarrow \P^{(\mu)} \Q^{(\mu^t)}$$
where the first map is inclusion (using the fact that there is a unique summand $\P^{(\l)}$ in $\P^{(\mu)} \P$ if $\l$ is obtained by adding a box to $\mu$ by ~\cite[Lemma 5]{CL1}) while the second map is adjunction. 
The composition $ \P^{(\l)} \Q^{(\l^t)} \rightarrow \P^{(\mu)} \Q^{(\mu^t)}$ defined above is unique up to rescaling.
However, an analogue of \cite[Prop. 4.7]{CLS}, can be used to show that there is a unique way (up to homotopy) to choose these multiples so that (\ref{eq:P1}) becomes a complex. The differentials in (\ref{eq:P2}) are defined similarly (or equivalently by defining $\T^+$ as the adjoint of $\T^-$).

\begin{conj}[Categorical projectors]\label{conj:3}
We have $\T^+ \P^{(\l)} \cong 0 \cong \T^- \P^{(\l)}$ for any nontrivial partition $\l$. In particular, this implies that $\T^+$ and $\T^-$ are projectors.
\end{conj}

At the decategorified level any representation of $\h$ generated by highest weight vectors breaks up into a direct sum of copies of the Fock space (each of which is generated by a highest weight vector). Conjecture \ref{conj:3} implies that $\T^{\pm}$ can be used to project onto these different copies of categorified Fock space.

One might wonder why there are two projectors, $\T^-$ and $\T^+$ and what this might mean? One possibility is that $\T^-$ and $\T^+$ are actually isomorphic.

\begin{conj}\label{conj:4}
We have $\T^+ \cong \T^-$.
\end{conj}

Let us illustrate how this might be possible despite the fact that $\T^+$ is supported in cohomological degrees $\ge 0$ and $\T^-$ in degrees $\le 0$. Suppose $\Q^n = 0$ for $n \ge 2$. Then $\T^- \cong [\P \Q \rightarrow \id]$ while $\T^+ \cong [\id \rightarrow \P \Q]$. Subsequently
\begin{align*}
\T^- \T^+
&\cong [\P \Q \rightarrow \id \oplus \P \Q \P \Q \rightarrow \P \Q] \\
&\cong [\P \Q \rightarrow \id \oplus \P \Q \rightarrow \P \Q].
\end{align*}
Now, one can check that both maps are isomorphisms between the $\P \Q$ summands. Thus one can use Lemma \ref{lem:cancel} to cancel out the two left-hand $\P \Q$ to obtain $[\id \rightarrow \P \Q] \cong \T^+$. On the other hand one can cancel out the two right-hand $\P \Q$  to obtain $[\P \Q \rightarrow \id] \cong \T^-$. Hence $\T^+ \cong \T^-$. We checked that the same argument works to prove Conjecture \ref{conj:4} if $\Q^n=0$ for $n>2$. We suspect that this argument will work in general.

\section{The twisted Boson-Fermion correspondence}

\subsection{Twisted Heisenberg algebra $\th$}

The twisted Heisenberg algebra $\th$ we are interested in is an associative unital algebra with generators $h_{m/2}$ for $ n \in 2\Z+1$ and relations $[h_{\frac{n}{2}}, h_{\frac{m}{2}}] = \frac{n}{2} \delta_{n,-m}$. 

As with $\h$ there is another way to present $\th$ by using generators $p^{(m)}$ and $q^{(m)}$ defined as follows:
\begin{equation}\label{eq:def1}
\sum_{m \in \frac{1}{2}\N} p^{(2m)} z^{m} = \exp \left( \sum_{m \in \N+ \frac{1}{2}} \frac{h_{-m}}{m} z^m \right) \ \ \text{ and }  \ \ \sum_{m \in \frac{1}{2}\N} q^{(2m)} z^m = \exp \left( \sum_{m \in \N+\frac{1}{2}} \frac{h_{m}}{m} z^m \right).
\end{equation}

\begin{prop}\label{prop:3}
The Heisenberg algebra $\th$ is generated by $p^{(m)}$ and $q^{(m)}$, $m \in \N$ with relations
\begin{enumerate}
\item $ p^{(n)} p^{(m)} = p^{(m)} p^{(n)} $ and $q^{(n)} q^{(m)} = q^{(m)} q^{(n)}$ for any $m,n$,
\item $ \sum_{k=0}^n p^{(2k)} p^{(2n-2k)} = \sum_{k=0}^{n-1} p^{(2k+1)} p^{(2n-2k-1)} $ and similarly for $q$'s, 
\item $ q^{(n)} p^{(m)} = p^{(m)} q^{(n)} + \sum_{k \ge 1} 2 p^{(m-k)} q^{(n-k)}$.
\end{enumerate}
\end{prop}
\begin{proof}
The first set of relations are clear. The second relation follows from 
\begin{align*}
\sum_{n \in \mathbb{Z}_{\geq 0}} \sum_{k=0}^n (-1)^k p^{(2n-k)} p^{(k)} z^n &=
(\sum_{n \in \frac{1}{2} \mathbb{N}} p^{(2n)} z^n)(\sum_{n \in \frac{1}{2} \mathbb{N}} (-1)^{2n} p^{(2n)} z^n) \\
&= \exp(\sum_{n \in \mathbb{N} + \frac{1}{2}} \frac{h_{-n}}{n} z^n)\exp(-\sum_{n \in \mathbb{N} + \frac{1}{2}} \frac{h_{-n}}{n} z^n) \\
&=1.
\end{align*}
To see the third relation note that by ~\cite[Proposition 3.4.1]{FLM} we have
\begin{equation}
\label{vertexcomm}
[\sum_{{m} \in \frac{1}{2}\N} q^{(2m)} z_1^{m}][\sum_{{m} \in \frac{1}{2}\N} p^{(2m)} z_2^{m}]=[\sum_{{m} \in \frac{1}{2}\N} p^{(2m)} z_2^{m}][\sum_{{m} \in \frac{1}{2}\N} q^{(2m)} z_1^{m}] \left( \frac{1+z_2^{\frac{1}{2}}z_1^{\frac{1}{2}}}{1-z_2^{\frac{1}{2}}z_1^{\frac{1}{2}}} \right).
\end{equation}
It is not hard to check that 
$$\frac{1+z_2^{\frac{1}{2}}z_1^{\frac{1}{2}}}{1-z_2^{\frac{1}{2}}z_1^{\frac{1}{2}}} = 1+2z_2^{\frac{1}{2}}z_1^{\frac{1}{2}} + 2z_2z_1^{}+ 2z_2^{\frac{3}{2}} z_1^{\frac{3}{2}} + \cdots.$$
Substituting this expansion into \eqref{vertexcomm} gives the commutator relation.
\end{proof}

One could also define elements $ p^{(1^m)} $ and $ q^{(1^m)} $ by
$$\sum_{m \in \frac{1}{2}\N} (-1)^{2m} p^{(1^{2m})} z^{m} = \exp \left( - \sum_{m \in \N+ \frac{1}{2}} \frac{h_{-m}}{m} z^m \right) \ \ \text{ and }  \ \ \sum_{m \in \frac{1}{2}\N} (-1)^{2m} q^{(1^{2m})} z^m = \exp \left( - \sum_{m \in \N+\frac{1}{2}} \frac{h_{m}}{m} z^m \right).$$
However, it is not difficult to see that, unlike in the untwisted case, these give the same generators as equation (\ref{eq:def1}) -- meaning that $ p^{(1^m)}=p^{(m)} $ and $ q^{(1^m)}=q^{(m)}$. The reason for this is essentially that the coefficients of $z^m$ on the right sides are purely positive or purely negative sums of terms. 

Just like $\h$, the twisted Heisenberg $\th$ also has the irreducible Fock space representation $V_{Fock}$. 

\subsection{The twisted correspondence}

The twisted Clifford algebra $\cl^t$ has generators $\phi_i$ ($i \in \Z$) with relations 
\begin{equation}\label{eq:tcl1}
\phi_i \phi_j + \phi_j \phi_i = 2 \delta_{i,-j}.
\end{equation}

It will also be convenient to consider the algebra $\tcl^t$ which has generators $\phi_i$ and $\phi_i^*$ ($i \in \Z$) and relations
\begin{equation}\label{eq:tcl2}
\phi_i \phi_j + \phi_j \phi_i = 2 \delta_{i,-j} \ \ \phi_i^* \phi_j^* + \phi_j^* \phi_i^* = 2 \delta_{i,-j} \ \ \text{ and } \ \ \phi_i \phi_j^* + \phi_j^* \phi_i = 2 \delta_{i,j}. 
\end{equation}
Note that $\tcl^t$ contains $\cl^t$ but we also have a surjective map $\tcl^t \rightarrow \cl^t$ given by $\phi_i \mapsto \phi_i$ and $\phi_i^* \mapsto \phi_{-i}$. The algebra $\tcl^t$ appears in \cite{J} where it is studied in the decategorified setup of our story. 

\begin{theorem}[Twisted Boson-Fermion correspondence]\label{thm:B}
One can define an action of $\cl^t$ on $V = V_{Fock}$ by $\phi_i(v) := C_i(v)$ where 
\begin{equation}\label{eq:C3} 
C_i = \begin{cases}
\sum_{k \ge 0} (-1)^{k} p^{(k)} q^{(i+k)} & \text{ if } i \ge 0 \\ 
\sum_{k \ge 0} (-1)^{i+k} p^{(-i+k)} q^{(k)} & \text{ if } i \le 0
\end{cases} 
\end{equation}
\end{theorem}
\begin{proof}
This is essentially proved in \cite[eq. 14]{An}. The only difference is that in their definition of $\cl^t$, relation (\ref{eq:tcl1}) above is actually $\phi_i \phi_j + \phi_j \phi_i = 2 (-1)^i \delta_{i,-j}$. It is easy to see that this is equivalent since one can rescale each $\phi_i$ for $i > 0$ by $(-1)^i$. The advantage of (\ref{eq:tcl1}) is that we can take $V = V_{Fock}$ whereas in \cite{An} one takes $V = \k[e^{\pm \alpha}]/(e^{2 \alpha}-1) \otimes_\k V_{Fock}$. 
\end{proof}

Note that, in contrast to the untwisted case from Theorem \ref{thm:A}, we only have $\C_i$ instead of $\C_i^+$ and $\C_i^-$ since we do not have operators $p^{(1^n)}$ and $q^{(1^n)}$ in the twisted case. 

The algebra $\cl^t$ is usually referred to as the Clifford algebra of type $B_{\infty}$. For more details on the twisted Boson-Fermion correspondence see \cite{F} or \cite{DJKM}. In \cite{An} it is shown that this correspondence can be realized as an isomorphism of twisted vertex operator algebras. 

\section{The categorical twisted Boson-Fermion correspondence}

\subsection{The Heisenberg category $\tH$}\label{sec:tH}

We now introduce the twisted analogue of $\H$. The category $\tH$ is a $\k$-linear, $\Z/2\Z$-graded idempotent complete monoidal category, generated by objects $\P$ and $\Q$. The diagrammatic conventions are the same as those for $\H$ so that, in particular, $\P$ and $\Q$ are denoted by an upward and downward pointing strand.

The morphisms in $\tH$ are generated by diagrams as in (\ref{eq:maps}) but there are extra morphisms $\P \rightarrow \P \{1\}$ and $\Q \rightarrow \Q \{1\}$ where $\{1\}$ denotes the $\Z/2\Z$ grading shift. This extra map is illustrated by a hollow circle like this
$$ \begin{tikzpicture}[>=stealth]
\draw [->](0,0) -- (0,1) [thick];
\draw [black](0,.4) circle (2pt);
\end{tikzpicture} $$

These maps satisfy the same relations as in (\ref{eq:rel1}) and (\ref{eq:rel3}) and the right hand relation in (\ref{eq:rel2}). The left hand relation in (\ref{eq:rel2}) becomes

\begin{equation} \label{eq:rel2B}
\begin{tikzpicture}[>=stealth]
\draw (0,0) .. controls (1,1) .. (0,2)[<-][thick];
\draw (1,0) .. controls (0,1) .. (1,2)[->] [thick];
\draw (1.5,1) node {=};
\draw (2,0) --(2,2)[<-][thick];
\draw (3,0) -- (3,2)[->][thick];
\draw (3.5,1) node{$-$};
\draw (4,1.75) arc (180:360:.5) [thick];
\draw (4,2) -- (4,1.75) [thick];
\draw (5,2) -- (5,1.75) [thick][<-];
\draw (5,.25) arc (0:180:.5) [thick];
\draw (5,0) -- (5,.25) [thick];
\draw (4,0) -- (4,.25) [thick][<-];

\draw (5.5,1) node{$-$};
\draw (6,1.75) arc (180:360:.5) [thick];
\draw (6,2) -- (6,1.75) [thick];
\draw (7,2) -- (7,1.75) [thick][<-];
\draw (7,.25) arc (0:180:.5) [thick];
\draw (7,0) -- (7,.25) [thick];
\draw (6,0) -- (6,.25) [thick][<-];
\draw [black] (7,.1) circle (2pt);
\draw [black] (6,1.75) circle (2pt);
\end{tikzpicture}
\end{equation}

Finally, we have the following extra relations involving the new morphism:
\begin{equation}\label{eq:t1}
\begin{tikzpicture}[>=stealth]
\draw (0,0) -- (1,1)[->][thick];
\draw (1,0) -- (0,1)[->][thick];
\draw [black] (.25,.25) circle (2pt);
\draw (1.5,.5) node{=};
\draw (2,0) -- (3,1)[->][thick];
\draw (3,0) -- (2,1)[->][thick];
\draw [black] (2.75, .75) circle (2pt);

\draw (6,0) -- (7,1)[->][thick];
\draw (7,0) -- (6,1)[->][thick];
\draw [black] (6.25,.75) circle (2pt);
\draw (7.5,.5) node{=};
\draw (8,0) -- (9,1)[->][thick];
\draw (9,0) -- (8,1)[->][thick];
\draw [black] (8.75, .25) circle (2pt);
\end{tikzpicture}
\end{equation}

\begin{equation}\label{eq:t2}
\begin{tikzpicture}[>=stealth]
\draw (0,0) -- (0,-.5) [->][thick];
\draw [black] (0,-.25) circle (2pt);
\draw (1,0) -- (1,-.5) [thick];
\draw (0,0) arc (180:0:.5) [thick];
\draw (1.5,0) node {=};
\draw (1.8,0) node {$-$};
\draw (2,0) -- (2,-.5)[->] [thick];
\draw [black] (3,-.25) circle (2pt);
\draw (3,0) -- (3,-.5) [thick];
\draw (2,0) arc (180:0:.5) [thick];

\draw (5,0) -- (5,-.5) [thick];
\draw [black] (5,-.25) circle (2pt);
\draw (6,0) -- (6,-.5)[->] [thick];
\draw (5,0) arc (180:0:.5) [thick];
\draw (6.5,0) node {=};
\draw (7,0) -- (7,-.5) [thick];
\draw [black] (8,-.25) circle (2pt);
\draw (8,0) -- (8,-.5) [->][thick];
\draw (7,0) arc (180:0:.5) [thick];
\end{tikzpicture}
\end{equation}

\begin{equation}\label{eq:t3}
\begin{tikzpicture}[>=stealth]
\draw (0,0) -- (0,.5) [thick];
\draw [black] (0,.25) circle (2pt);
\draw (1,0) -- (1,.5) [->][thick];
\draw (0,0) arc (180:360:.5) [thick];
\draw (1.5,0) node {=};
\draw (2,0) -- (2,.5)[thick];
\draw [black] (3,.25) circle (2pt);
\draw (3,0) -- (3,.5) [->][thick];
\draw (2,0) arc (180:360:.5) [thick];

\draw (5,0) -- (5,.5)[->] [thick];
\draw [black] (5,.25) circle (2pt);
\draw (6,0) -- (6,.5) [thick];
\draw (5,0) arc (180:360:.5) [thick];
\draw (6.5,0) node {=};
\draw (6.8,0) node {$-$};
\draw (7,0) -- (7,.5)[->] [thick];
\draw [black] (8,.25) circle (2pt);
\draw (8,0) -- (8,.5) [thick];
\draw (7,0) arc (180:360:.5) [thick];
\end{tikzpicture}
\end{equation}

\begin{equation}\label{eq:t4}
\begin{tikzpicture}[>=stealth]
\draw (2.5,-1) node{$-$};
\draw (0,0) -- (0,-2)[<-][thick];
\draw (.5,-1) node{$=$};
\draw (1,0) -- (1,-2)[<-][thick];
\draw [black] (1, -.5) circle (2pt);
\draw [black] (1, -1.5) circle (2pt);

\draw (3,0) -- (3,-2)[->][thick];
\draw (3.5,-1) node{$=$};
\draw (4,0) -- (4,-2)[->][thick];
\draw [black] (4, -.5) circle (2pt);
\draw [black] (4, -1.5) circle (2pt);

\draw [shift={+(6,-1)}](0,0) arc (180:360:0.5cm) [thick];
\draw [shift={+(6,-1)}][->](1,0) arc (0:180:0.5cm) [thick];
\draw [black](7,-1) circle (2pt);
\draw [shift={+(6,-1)}](1.75,0) node{$= 0$};
\end{tikzpicture}
\end{equation}

\begin{equation}\label{eq:t5}
\begin{tikzpicture}[>=stealth]
\draw (0,1) -- (0,3)[->][thick];
\draw (1,1) -- (1,3)[->][thick];
\draw [black] (0, 2.5) circle (2pt);
\draw [black] (1, 1.5) circle (2pt);
\draw (.5,2) node{$\cdots$};
\draw [shift={+(-1,1)}](2.5,1) node{$=$};
\draw [shift={+(-2.8,1)}](4.8,1) node{$-$};
\draw [shift={+(-2.8,1)}](5,0) -- (5,2)[->][thick];
\draw [shift={+(-2.8,1)}](6,0) -- (6,2)[->][thick];
\draw [shift={+(-2.8,1)}][black] (5, .5) circle (2pt);
\draw [shift={+(-2.8,1)}][black] (6, 1.5) circle (2pt);
\draw [shift={+(-2.8,1)}](5.5,1) node{$\cdots$};
\end{tikzpicture}
\end{equation}

This last relation states that dots far apart supercommute which is something characteristic of superalgebras. If $\A$ is an object in $\tH$ we will use $\A \{s_1,\dots,s_k\}$ as short-hand for $\A \{s_1\} \oplus \dots \oplus \A \{s_k\}$. 

\subsection{Sergeev algebras}\label{sec:sergeev}

In contrast to the untwisted case where only the symmetric group algebra $\k[S_n]$ acts on $\Q^n$, in this case we also have the dots. These dots generate an action of the Clifford algebra $\Cliff_n$ which has generators $c_1, \dots, c_n$ with relations $ c_i^2 = -1 $ and $ c_i c_j = - c_j c_i$ for $ i\ne j$. 

More precisely, $\k[S_n]$ still acts on $\Q^n$ (by crossings) and we find that the semi-direct product 
$$\Se_n := \Cliff_n \rtimes \k[S_n]$$
acts on $\Q^n$ where $S_n$ acts on $\Cliff_n$ by permuting its generators. The algebra $\Se_n$ is called the Sergeev algebra. It is $\Z/2\Z$-graded where $|c_i|=1$ and $|s_i|=0$ (here $s_1, \dots, s_{n-1}$ are the standard generators of $S_n$ and $|\cdot|$ denotes the degree). Notice that on $\P^n$ we have an action of the opposite Sergeev algebra $\Se^{opp}$. Recall that the opposite of a superalgebra is defined by $a*b=(-1)^{|a||b|} ba$ which explains why two hollow dots on an upward strand are equal to $+1$ instead of $-1$. 

The super representation theory of $ \Se_n $ is equivalent to the projective representation theory of $S_n$. Thus finite dimensional irreducible super representations of $\Se_n$ are indexed by strict partitions of $n$. 

In \cite{N} Nazarov constructed quasi-idempotents in $\Se_n$ indexed by strict partitions.  An element $x \in \Se_n$ is quasi-idempotent if $ x^2=kx$ for some non-zero $k \in \k$.  His methods parallel Cherednik's construction of idempotents in $\k[S_n]$ using the degenerate affine Hecke algebra. We will use a different construction of quasi-idempotents $e_\l$ given later by Sergeev in \cite{Ser}. If $\l = (n)$ the element $e_{(n)}$ has the familiar expression $e_{(n)} = \frac{1}{n!} \sum_{w \in S_n} w$. It is easy to check that $e_{(n)}^2=e_{(n)}$ and $s_k e_{(n)} = e_{(n)} = e_{(n)} s_k$ for $1 \le k \le n-1$. 

One may also wonder what happens with the idempotent $e_{(1^n)} = \frac{1}{n!} \sum_{w \in S_n} (-1)^{l(w)} w$. It turns out this idempotent leads to equivalent representations. 

\begin{lemma}\label{lem:idempotent}
There is an isomorphism of $\Se_n$-modules $f: \Se_n e_{(n)} \cong \Se_n e_{(1^n)} \{n\}$.
\end{lemma}
\begin{proof}
Define $f: \Se_n e_{(n)} \rightarrow \Se_n e_{(1^n)} \{n\}$ by
$$f(c_1^{\epsilon_1} \cdots c_n^{\epsilon_n} e_{(n)}) = (-1)^{\epsilon_1+2\epsilon_2+ \cdots + n\epsilon_n} c_1^{\epsilon_1+1} \cdots c_n^{\epsilon_n+1} e_{(1^n)}$$
where $\epsilon_i \in \{ 0,1 \}$. It is easy to check that this is an isomorphism of $\Se_n$-modules.
\end{proof}

The definition of $\tH$, just like that of $\H$ from \cite{Kh1}, involves taking the idempotent closure. This means that the idempotent $e_{(n)}$ gives us objects $\P^{(n)}$ and $\Q^{(n)}$. Note that by Lemma \ref{lem:idempotent} we get isomorphic objects if we consider $\P^{(1^n)}$ and $\Q^{(1^n)}$.  The objects $ \P^{(n)} $ and $ \Q^{(n)} $ are depicted diagrammatically by rectangles labeled by $\lambda$ just as in the case of the category $\H$.

\begin{lemma}\label{lem:PP}
If $m \le n$ then as $\Se_{m+n}$-modules we have an isomorphism
$$\Se_{m+n} \otimes_{\Se_m \times \Se_n}  (\Se_m e_{(m)} \times \Se_n e_{(n)}) \cong \bigoplus_{k=0}^m (\Se_{m+n} e_{(n+k,m-k)} \{0,1\})^{\oplus d_k}$$
where if $m < n$ then $d_k=1$ if $k=0,m$ and $d_k=2$ otherwise while if $m=n$ then $d_k=1$ if $k=m$, $d_k=0$ if $k=0$, and $d_k=2$ otherwise. 
\end{lemma}
\begin{proof}
Ignoring the grading shifts, the decomposition follows from multiplication of Schur Q-functions for two single row partitions.  For more details on these formulas see for example ~\cite[section 3.4]{WW}.

There is a degree zero isomorphism $ f \colon \Se_m e_{(m)} \rightarrow \Se_m e_{(m)} \{1 \}$ where
$$ c_1^{\epsilon_1} \cdots c_m^{\epsilon_m} e_{(m)} \mapsto c_1^{\epsilon_1} \cdots c_m^{\epsilon_m} (c_1+\cdots+c_m) e_{(m)}. $$ 
This gives rise to a degree zero isomorphism
$$ F \colon \Se_{m+n} \otimes_{\Se_m \times \Se_n}  (\Se_m e_{(m)} \times \Se_n e_{(n)}) \rightarrow \Se_{m+n} \otimes_{\Se_m \times \Se_n}  (\Se_m e_{(m)} \times \Se_n e_{(n)}) \{1\}$$
where $ w \otimes (v_1 \otimes v_2) \mapsto w \otimes (f(v_1) \otimes v_2) $.
Since the induced module is evenly isomorphic to its shift, every appearance of a simple module in its decomposition must come paired with its shifted module.
\end{proof}

\begin{theorem}\label{thm:2}
Inside $\tH$ we have the following relations
\begin{enumerate}
\item $\P^{(\l)}$ and $\P^{(\mu)}$ commute for any strict partitions $\l,\mu$,
\item $\Q^{(\l)}$ and $\Q^{(\mu)}$ commute for any strict partitions $\l,\mu$,
\item 
If $m < n$ then
$$\P^{(m)} \P^{(n)} \cong \P^{(m+n)} \{0,1 \} \oplus \P^{(n,m)} \{0,1\} \oplus \left[ \P^{(m+n-1,1)} \{0,1\} \oplus \dots \oplus \P^{(n+1,m-1)} \{0,1\} \right]^{\oplus 2}$$
while if $m=n$ then 
$$\P^{(m)} \P^{(n)} \cong \P^{(m+n)} \{0,1 \} \oplus \left[ \P^{(m+n-1,1)} \{0,1\} \oplus \dots \oplus \P^{(n+1,m-1)} \{0,1\} \right]^{\oplus 2}.$$
Similar relations hold for $\Q^{(m)} \Q^{(n)}$.
\item $\Q^{(n)} \P^{(m)} \cong \P^{(m)} \Q^{(n)} \bigoplus_{k \ge 1} \left( \P^{(m-k)} \Q^{(n-k)} \oplus \P^{(m-k)} \Q^{(n-k)} \{ 1 \} \right)$.
\end{enumerate}
\end{theorem}
\begin{proof}
The first two relations are obvious. The third relation follows from Lemma \ref{lem:PP}. To prove the last relation define maps 
$$A_{k,c}: \Q^{(n)} \P^{(m)} \rightarrow \P^{(m-k)} \Q^{(n-k)}  \{ c \} \ \ \text{ and } \ \ B_{k,c}: \P^{(m-k)} \Q^{(n-k)} \rightarrow \Q^{(n)} \P^{(m)} \{ c \}$$
(for $c=0,1$) as follows: 
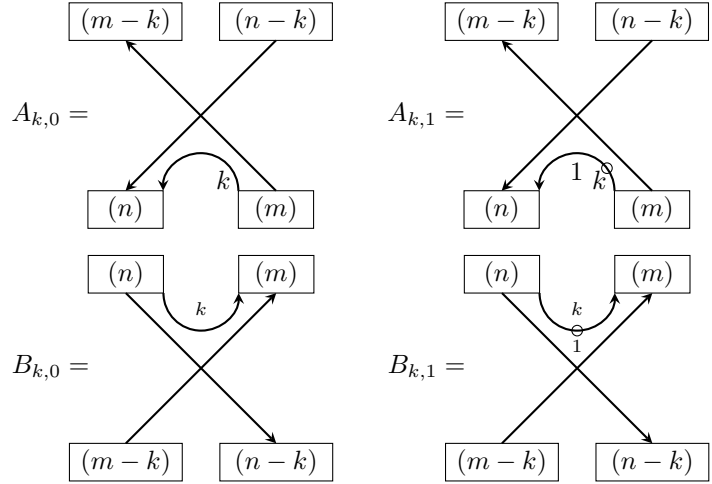
\begin{equation*}
\begin{tikzpicture}[>=stealth]
\draw (4.5,1.5) node {$A_{k,0} =$};
\draw (5,0) rectangle (6,.5);
\draw (5.5,.25) node {$(n)$};
\draw (4.75,2.5) rectangle (6.25,3);
\draw (5.5,2.75) node {${(m-k)}$};
\draw (5.5,.5) -- (7.5,2.5) [<-][thick];
\draw (7.5,.5) -- (5.5,2.5) [->][thick];
\draw (7,0) rectangle (8,.5);
\draw (7.5,.25) node {${(m)}$};
\draw (6.75,2.5) rectangle (8.25,3);
\draw (7.5,2.75) node {${(n-k)}$};
\draw (6,.5) arc (180:0:0.5cm)[<-] [thick];
\draw (6.8,.65) node{$k$};

\draw (9.5,1.5) node {$A_{k,1} =$};
\draw (10,0) rectangle (11,.5);
\draw (10.5,.25) node {$(n)$};
\draw (9.75,2.5) rectangle (11.25,3);
\draw (10.5,2.75) node {${(m-k)}$};
\draw (10.5,.5) -- (12.5,2.5) [<-][thick];
\draw (12.5,.5) -- (10.5,2.5) [->][thick];
\draw (12,0) rectangle (13,.5);
\draw (12.5,.25) node {${(m)}$};
\draw (11.75,2.5) rectangle (13.25,3);
\draw (12.5,2.75) node {${(n-k)}$};
\draw (11,.5) arc (180:0:0.5cm)[<-] [thick];
\draw (11.8,.65) node{$k$};
\draw [black] (11.9,.8) circle (2pt);
\draw (11.5,.75) node{$1$};
\end{tikzpicture}
\end{equation*}
\begin{equation*}
\begin{tikzpicture}[>=stealth]
\draw (4.5,1.5) node {$B_{k,0} =$};
\draw (4.75,0) rectangle (6.25,.5);
\draw (5.5,2.75) node {$(n)$};
\draw (5,2.5) rectangle (6,3);
\draw (5.5,.25) node {${(m-k)}$};
\draw (5.5,.5) -- (7.5,2.5) [->][thick];
\draw (7.5,.5) -- (5.5,2.5) [<-][thick];
\draw (6.75,0) rectangle (8.25,.5);
\draw (7.5,2.75) node {${(m)}$};
\draw (7,2.5) rectangle (8,3);
\draw (7.5,.25) node {${(n-k)}$};
\draw (6,2.5) arc (-180:0:0.5cm)[->] [thick];
\draw (6.5,2.3) node{$\scriptstyle{k}$};

\draw (9.5,1.5) node {$B_{k,1} =$};
\draw (9.75,0) rectangle (11.25,.5);
\draw (10.5,2.75) node {$(n)$};
\draw (10,2.5) rectangle (11,3);
\draw (10.5,.25) node {${(m-k)}$};
\draw (10.5,.5) -- (12.5,2.5) [->][thick];
\draw (12.5,.5) -- (10.5,2.5) [<-][thick];
\draw (11.75,0) rectangle (13.25,.5);
\draw (12.5,2.75) node {${(m)}$};
\draw (12,2.5) rectangle (13,3);
\draw (12.5,.25) node {${(n-k)}$};
\draw (11,2.5) arc (-180:0:0.5cm)[->] [thick];
\draw (11.5,2.3) node{$\scriptstyle{k}$};
\draw [black] (11.5,2) circle (2pt);
\draw (11.5,1.8) node{$\scriptstyle{1}$};
\end{tikzpicture}
\end{equation*}
where the label $k$ means there are $k$ strands and the label $1$ for the hollow dot means put a single hollow dot on the upper most (resp. lowermost) arc.

We claim that $ \gamma_{m,n}^{k,c} A_{k,c}$ and $B_{k,c}$ for carefully chosen non-zero constants $ \gamma_{m,n}^{k,c} $ give rise to inverse maps. 
The proof of this is very similar to that of the third relation in Theorem \ref{thm:1}. One key point to note is that two hollow dots sandwiched between $\P^{(n)}$ (for $n \ge 2$) is equal to zero. The reason for this is illustrated diagrammatically below (when $n=2$). It uses the fact that dots supercommute. 
\begin{equation}\label{eq:zero}
\begin{tikzpicture}[>=stealth]
\draw (-.25,1) rectangle (1.75,1.5);
\draw (.75,1.25) node {$(2)$};

\draw [shift = {+(0,2)}](-.25,1) rectangle (1.75,1.5);
\draw [shift = {+(0,2)}](.75,1.25) node {$(2)$};

\draw (.5,1.5) -- (.5,3) [->][thick];
\draw (1,1.5) -- (1,3) [->][thick];
\draw [black] (.5,2.0) circle (3pt);
\draw [black] (1,2.5) circle (3pt);
\draw (2,2.25) node {$= $};

\draw [shift = {+(3,0)}](-.25,1) rectangle (1.75,1.5);
\draw [shift = {+(3,0)}](.75,1.25) node {$(2)$};

\draw [shift = {+(3,0)}][shift = {+(0,2)}](-.25,1) rectangle (1.75,1.5);
\draw [shift = {+(3,0)}][shift = {+(0,2)}](.75,1.25) node {$(2)$};

\draw [shift = {+(3,0)}](.5,1.5) -- (1,2) [thick];
\draw [shift = {+(3,0)}](1,1.5) -- (.5,2) [thick];

\draw [shift = {+(3,0)}](.5,2) -- (.5,3) [->][thick];
\draw [shift = {+(3,0)}](1,2) -- (1,3) [->][thick];
\draw [shift = {+(3,0)}][black] (.5,2.25) circle (3pt);
\draw [shift = {+(3,0)}][black] (1,2.5) circle (3pt);
\draw [shift = {+(3,0)}](2,2.25) node {$= \ \ -$};

\draw [shift = {+(6,0)}](-.25,1) rectangle (1.75,1.5);
\draw [shift = {+(6,0)}](.75,1.25) node {$(2)$};
\draw [shift = {+(6,0)}][shift = {+(0,2)}](-.25,1) rectangle (1.75,1.5);
\draw [shift = {+(6,0)}][shift = {+(0,2)}](.75,1.25) node {$(2)$};
\draw [shift = {+(6,0)}](.5,2.5) -- (1,3) [->][thick];\draw [shift = {+(6,0)}](1,2.5) -- (.5,3) [->][thick];

\draw [shift = {+(6,0)}](.5,1.5) -- (.5,2.5) [thick];
\draw [shift = {+(6,0)}](1,1.5) -- (1,2.5) [thick];
\draw [shift = {+(6,0)}][black] (.5,2) circle (3pt);
\draw [shift = {+(6,0)}][black] (1,2.25) circle (3pt);
\draw [shift = {+(6,0)}](2,2.25) node {$= \ \ -$};
\draw  [shift = {+(9,0)}](-.25,1) rectangle (1.75,1.5);
\draw [shift = {+(9,0)}] (.75,1.25) node {$(2)$};

\draw  [shift = {+(9,0)}][shift = {+(0,2)}](-.25,1) rectangle (1.75,1.5);\draw  [shift = {+(9,0)}][shift = {+(0,2)}](.75,1.25) node {$(2)$};
\draw  [shift = {+(9,0)}](.5,1.5) -- (.5,3) [->][thick];
\draw  [shift = {+(9,0)}](1,1.5) -- (1,3) [->][thick];
\draw  [shift = {+(9,0)}][black] (.5,2.0) circle (3pt);
\draw  [shift = {+(9,0)}][black] (1,2.5) circle (3pt);

\end{tikzpicture}
\end{equation}
This explains why in the definitions of maps $A_{k,c}$ and $B_{k,c}$ above one puts at most one hollow dot on the caps and cups ({\it{i.e.}} $c=0,1$). 
\end{proof}

\subsection{A categorification of Fock space}\label{sec:fock}

We now explain how to lift the action of $\th$ on the Fock space. This is completely analogous to the untwisted case so we only sketch the construction. In that case the categorified Fock space is $\V_{Fock} = \oplus_{n \ge 0} S_n\mod$. In our case we take $\V_{Fock} = \oplus_{n \ge 0} \Se_n\mod$ consisting of $\Z/2\Z$-graded finite dimensional left $\Se_n$-modules. 

Since $\Se_n$ is naturally a subalgebra of $\Se_{n+1}$ we can give $\Se_{n+1}$ the structure of a $(\Se_{n+1},\Se_n)$-bimodule as well as a $(\Se_n,\Se_{n+1})$-bimodule. This defines two functors 
\begin{align*}
\P(n): \Se_n\mod \rightarrow \Se_{n+1}\mod \ \ \ & M \mapsto \Se_{n+1} \otimes_{\Se_n} M \\
\Q(n): \Se_{n+1}\mod \rightarrow \Se_n\mod \ \ \ & N \mapsto \Se_{n+1} \otimes_{\Se_{n+1}} N.
\end{align*}

It remains to define the natural transformations between these functors. 
The functors and natural transformations which follow were first considered in ~\cite{BK}.
First, the hollow dots on an upward and downward strand are defined by
\begin{align*}
\P(n) \rightarrow \P(n+1) \{1\} \ \ \ & x \mapsto (-1)^{|x|} xc_{n+1} \\
\Q(n) \rightarrow \Q(n+1) \{1\} \ \ \ & x \mapsto c_{n+1}x
\end{align*}
The reason for the extra $(-1)^{|x|}$ factor in the first expressions above is that an odd morphism of super-bimodules must be right linear but left antilinear. 

The composition $\P(n+1) \circ \P(n)$ is induced by $\Se_{n+2}$ viewed as a $(\Se_{n+2},\Se_n)$-bimodule. The crossing morphism is given by the map $\Se_{n+2} \rightarrow \Se_{n+2}$, $x \mapsto xs_{n+1}$. Likewise, $\Q(n) \circ \Q(n+1)$ is also induced by $\Se_{n+2}$ with the downward crossing given by $x \mapsto s_{n+1}x$. 

A right-pointing crossing is a map $\P(n-1) \circ \Q(n-1) \rightarrow \Q(n) \circ \P(n)$. Now $\P(n-1) \circ \Q(n-1)$ corresponds to the bimodule $\Se_n \otimes_{\Se_{n-1}} \Se_n$ while $\Q(n) \circ \P(n)$ to the bimodule $\Se_{n+1}$. The corresponding map is then 
$$\Se_n \otimes_{\Se_{n-1}} \Se_n \rightarrow \Se_{n+1} \ \ \ x \otimes y \mapsto x s_n y.$$
Next, an element of $\Se_{n+1}$ either contains no $s_n$ or can be written as $x s_n y$ where $x,y$ contain no $s_n$. The left-pointing crossing is a map $\Q(n) \circ \P(n) \rightarrow \P(n-1) \circ \Q(n-1)$ given by a map 
$$\Se_{n+1} \rightarrow \Se_n \otimes_{\Se_{n-1}} \Se_n \ \ \ x \mapsto 0, \ \ x s_n y \mapsto x \otimes y.$$
Finally, the four adjunction maps are given as follows
\begin{align*}
\Q(n) \circ \P(n) \rightarrow \id: \ \ & \Se_{n+1} \rightarrow \Se_n, \text{ natural projection mapping } s_n, c_{n+1} \mapsto 0 \\
\id \rightarrow \Q(n) \circ \P(n): \ \ & \Se_n \rightarrow \Se_{n+1}, \text{ natural inclusion } \\
\P(n) \circ \Q(n) \rightarrow \id: \ \ & \Se_{n+1} \otimes_{\Se_n} \Se_{n+1} \rightarrow \Se_{n+1}, \text{ multiplication map } x \otimes y \mapsto xy \\
\id \rightarrow \P(n) \circ \Q(n): \ \ & \Se_{n+1} \rightarrow \Se_{n+1} \otimes_{\Se_n} \Se_{n+1}, \\ 
& 1 \mapsto \sum_{i=1}^{n+1} (s_i \cdots s_n \otimes s_n \cdots s_i - s_i \cdots s_n c_{n+1} \otimes c_{n+1} s_n \cdots s_i).
\end{align*}

What remains is proving all the relations encoded in the diagrams from section \ref{sec:tH} together with the isotopy relations. The strand isotopy relations amount to proving the following relations

\begin{equation}
\label{isotopy1}
\begin{tikzpicture}
\draw[thick] (0,0) to (0,1) arc(180:0:.5) arc(180:360:.5) to (2,2)[->];
\draw (2.5,1) node {$=$};
\draw [thick] (3,0) -- (3,2)[->];
\draw (3.5,1) node {$=$};
\draw[thick] (4,2)[<-] to (4,1) arc(180:360:.5) arc(180:0:.5) to (6,0);
  
\draw[thick] (8,0) to (8,1)[<-] arc(180:0:.5) arc(180:360:.5) to (10,2);
\draw (10.5,1) node {$=$};
\draw [thick] (11,0) -- (11,2)[<-];
\draw (11.5,1) node {$=$};
\draw[thick] (12,2) to (12,1) arc(180:360:.5) arc(180:0:.5) to (14,0)[->];
\end{tikzpicture}
\end{equation}

\begin{equation}
\label{isotopy4}
\begin{tikzpicture}
   \draw[->,thick] (0.5,1) .. controls (0.5,0) and (1.5,0) .. (1.5,1);
  \draw[<-,thick] (1,1) to (0.5,0);
  \draw (2.25,0.5) node {=};
  \draw[->,thick] (3,1) .. controls (3,0) and (4,0) .. (4,1);
  \draw[<-,thick] (3.5,1) to (4,0);
   \draw[->,thick] (6.5,1) .. controls (6.5,0) and (7.5,0) .. (7.5,1);
  \draw[->,thick] (7,1) to (6.5,0);
  \draw (8.25,0.5) node {=};
  \draw[->,thick] (9,1) .. controls (9,0) and (10,0) .. (10,1);
  \draw[->,thick] (9.5,1) to (10,0);
\end{tikzpicture}
\end{equation}

\begin{equation}
\label{isotopy5}
\begin{tikzpicture}
  \draw[<-,thick] (0.5,0) .. controls (0.5,1) and (1.5,1) .. (1.5,0);
  \draw[<-,thick] (1.5,1) to (1,0);
  \draw (2.25,0.5) node {=};
  \draw[<-,thick] (3,0) .. controls (3,1) and (4,1) .. (4,0);
  \draw[<-,thick] (3,1) to (3.5,0);
  \draw[<-,thick] (6.5,0) .. controls (6.5,1) and (7.5,1) .. (7.5,0);
  \draw[->,thick] (7.5,1) to (7,0);
  \draw (8.25,0.5) node {=};
  \draw[<-,thick] (9,0) .. controls (9,1) and (10,1) .. (10,0);
  \draw[->,thick] (9,1) to (9.5,0);
\end{tikzpicture}
 \end{equation}

These relations, along with those from section \ref{sec:tH}, can be checked by direct calculations (akin to those calculations from \cite{Kh1} and \cite{CL1}). We include one instance of such a computation, namely the left hand relation from (\ref{eq:t2}). Both sides of that equality correspond to a composition 
$$\Q(n) \P(n) \rightarrow \Q(n) \P(n) \{1\} \rightarrow \id \{1\}$$
which is a sequence of maps 
$$\Se_{n+1} \otimes_{\Se_{n+1}} \Se_{n+1} \rightarrow \Se_{n+1} \otimes_{\Se_{n+1}} \Se_{n+1} \{1\} \rightarrow \Se_n \{1\}.$$
Consider an element $x \otimes y \in \Se_{n+1} \otimes_{\Se_{n+1}} \Se_{n+1}$. If a reduced expression of $xy$ contains $s_n$ or if does not contain $c_{n+1}$ then both sides of the left equality in (\ref{eq:t2}) map $x \otimes y$ to zero. On the other hand, if $xy$ does not contain $s_n$ but does contain $c_{n+1}$ then the left side of the left equality in (\ref{eq:t2}) acts by
$$x \otimes y \mapsto c_{n+1}x \otimes y \mapsto c_{n+1}xy$$
while the right side acts by 
$$x \otimes y \mapsto (-1)^{|x|+|y|} x \otimes yc_{n+1} \mapsto (-1)^{|x|+|y|} xyc_{n+1} = - c_{n+1} xy$$
where the last equality follows since $xy$ does not contain $s_n$ but does contain one copy of $c_{n+1}$. The result follows. 

\begin{corollary}
There exists an injective map $\th \rightarrow K_0(\tH)$. 
\end{corollary}
\begin{proof}
Relations (1) and (3) from Proposition \ref{prop:3} follow immediately from Theorem \ref{thm:2}. Moreover, it is not hard to show that relation (2) from Proposition \ref{prop:3} follows from relation (3) in Theorem \ref{thm:2} after a little manipulation. This gives us a map $\th \rightarrow K_0(\tH)$ by taking $p^{(n)} \mapsto \P^{(n)}$ and $q^{(n)} \mapsto \Q^{(n)}$ and extending multiplicatively. 

The fact that this map is injective follows using the same argument as in \cite[Sec. 3.3]{Kh1} using the action of $\tH$ on $\oplus_{n \ge 0} \Se_n\mod$. 
\end{proof}

Note that, as in the untwisted case, we do not know surjectivity of the map from the Corollary above. 

\subsection{Complexes}

We work once again in the homotopy category $\Kom(\tH)$. We define 
\begin{equation} \label{eq:cpxC3} 
\C^-_i := \begin{cases}
\left( \dots \rightarrow \P^{(k)} \Q^{(i+k)} \{ k \} \rightarrow \dots \rightarrow \P \Q^{(i+1)} \{ 1 \} \rightarrow \Q^{(i)} \right) & \text{ if } i \ge 0 \\
\left( \dots \rightarrow \P^{(-i+k)} \Q^{(k)} \{k\} \rightarrow \dots \rightarrow \P^{(-i+1)} \Q \{1\} \rightarrow \P^{(-i)} \right) [-i]  & \text{ if } i \le 0
\end{cases}
\end{equation}
where the right most term is in cohomological degree zero. The differential is given by 
\begin{equation} \label{eq:diff2}
\begin{tikzpicture}[>=stealth]
\draw (9.75,0) rectangle (11.25,.5);
\draw (10.5,.25) node {$(k)$};
\draw (9.75,1.5) rectangle (11.25,2);
\draw (10.5,1.75) node {$(k-1)$};
\draw (10.5,.5) -- (10.5,1.5) [->][thick];
\draw (12.5,.5) -- (12.5,1.5) [<-][thick];
\draw (11.75,0) rectangle (13.25,.5);
\draw (12.5,.25) node {$(i+k)$};
\draw (11.5,1.5) rectangle (13.5,2);
\draw (12.5,1.75) node {$(i+k-1)$};
\draw (11,.5) arc (180:0:0.5cm)[->] [thick];
\draw [black] (11.8,.9) circle (2pt);
\end{tikzpicture}
\end{equation}
where there is one cap with a hollow dot. Notice that without this hollow dot the differential would not square to zero (with the dot it squares to zero using the argument in (\ref{eq:zero})). Likewise, we define
\begin{equation} \label{eq:cpxC4} 
\C^+_i := \begin{cases}
\left( \P^{(i)} \rightarrow \P^{(i+1)} \Q \{1\} \rightarrow \dots \rightarrow \P^{(i+k)} \Q^{(k)} \{k\} \rightarrow \dots \right) & \text{ if } i \ge 0 \\
\left( \Q^{(-i)} \rightarrow \P \Q^{(-i+1)} \{1\} \rightarrow \dots \rightarrow \P^{(k)} \Q^{(-i+k)} \{k\} \rightarrow \dots \right) [i] & \text{ if } i \le 0
\end{cases}
\end{equation}
where the left most term is in cohomological degree zero and the differential is given by
\begin{equation} \label{eq:diff3}
\begin{tikzpicture}[>=stealth]
\draw (9.5,0) rectangle (11.5,.5);
\draw (10.5,1.75) node {$(i+k)$};
\draw (9.75,1.5) rectangle (11.25,2);
\draw (10.5,.25) node {$(i+k-1)$};
\draw (10.5,.5) -- (10.5,1.5) [->][thick];
\draw (12.5,.5) -- (12.5,1.5) [<-][thick];
\draw (11.75,0) rectangle (13.25,.5);
\draw (12.5,1.75) node {$(k)$};
\draw (11.75,1.5) rectangle (13.25,2);
\draw (12.5,.25) node {$(k-1)$};
\draw (11,1.5) arc (-180:0:0.5cm)[<-] [thick];
\draw [black] (11.9,1.2) circle (2pt);
\end{tikzpicture}
\end{equation}

\subsection{A categorical action of $\tcl^t$}

We will now define an action of $\tcl^t$ on $\V_{Fock}$. 

\begin{conj}[Categorical twisted Boson-Fermion correspondence]\label{conj:2}\ \\
On $\Kom(\V_{Fock})$ define functors $\Phi_i(\v) := \C_{i}^-(\v)$ and $\Phi_i^*(\v) := \C_{i}^+(\v)$. These functors satisfy the following relations 
\begin{enumerate}
\item $\Phi_i \Phi_j \cong \begin{cases} \Phi_j \Phi_i [-1]\{1\} &\text{ if } i < j, i+j \ne 0 \\ \Phi_j \Phi_i [1]\{1\} & \text{ if } i > j, i+j \ne 0 \end{cases}$ while $\Phi_i^2 \cong 0$ if $i \ne 0$ and likewise if you replace $\Phi_i$ with $\Phi_i^*$ and $\Phi_j$ with $\Phi_j^*$
\item \label{R3t} $\Phi_i \Phi_j^* \cong \begin{cases} \Phi_j^* \Phi_i [1]\{1\} & \text{ if } i < j \\ \Phi_j^* \Phi_i [-1]\{1\} & \text{ if } i > j \end{cases}$ 
\item \label{R4t} $\Cone(\id \rightarrow \Phi_i^* \Phi_i) \cong \Cone(\Phi_i \Phi_i^* \rightarrow \id) \{1\}$ 
\end{enumerate}
\end{conj}
\begin{remark}
It is easy to see that $\Phi_i^*$ is biadjoint to $\Phi_i$ just because $\P$ is biadjoint to $\Q$. The two maps that appear in the second relation above are the corresponding adjunctions maps. 
\end{remark}

As in the untwisted case it is easy to see that $\Phi_i^*$ is both left and right adjoint to $\Phi_i$. Notice that at the level of Grothendieck theory, complexes $\C_i^-$ and $\C_{-i}^+$ are equal (up to sign). Thus the action of $\tcl^t$ on Grothendieck groups descends to an action of $\cl^t$. 

The relations above are categorical analogues only of some of the relations in (\ref{eq:tcl2}). For instance, we have $\phi_0^2 = 1$ at the decategorified level but we do not have an analogous relation for $\Phi_0^2$. This is because $\Phi_0^2$ does not have an obvious simplification (in particular, $\Phi_0^2 \not\cong \id$ as illustrated in example \ref{sec:ex6} below). It is tempting to try to figure out these missing relations. 

\subsection{Examples: twisted case}\label{sec:extwisted}

We finish by illustrating a couple of computations.

\subsubsection{Example 3}
Suppose that $Q^n=0$ for $ n \ge 1$. Then $ \Phi_{-2} \Phi_0 \cong \P^{(2)}[2]$.  On the other hand
\begin{align}
\nonumber \Phi_{0} \Phi_{-2}
&\cong [\P^{(2)} \Q^{(2)}  \rightarrow \P\Q \{1 \} \rightarrow \id][\P^{(2)}][2] \\
\label{calc:11} &\cong \left[\P^{(2)} \oplus \P^{(2)} \{1\} \rightarrow \P \P \{1\} \oplus \P\P \rightarrow \P^{(2)} \right][2].
\end{align} 
One can check carefully that the first map in (\ref{calc:11}) maps $\P^{(2)} \oplus \P^{(2)} \{1\}$ isomorphically into $ \P\P\{1\}$. Using the cancellation Lemma ~\ref{lem:cancel} the complex (\ref{calc:11}) is homotopically equivalent to $\P\P \rightarrow \P^{(2)}$. This map is induced by the second map in (\ref{calc:11}) and one can again check that it is a projection. Thus $\Phi_0 \Phi_{-2} \cong \P^{(2)} [3]\{1\}$ which means that $ \Phi_0 \Phi_{-2} \cong \Phi_{-2} \Phi_0 [1]\{1\}$.

\subsubsection{Example 4}\label{sec:ex4} Assume again that $\Q^n = 0$ for $n > 1$ so that
\begin{align}
\nonumber \Phi_{-1}^* \Phi_{-1}
&\cong [\Q \rightarrow \P \Q^{(2)} \{1\}][\P^{(2)} \Q \{1\} \rightarrow \P] \\
\nonumber &\cong \left[ \Q \P^{(2)} \Q \{1\} \rightarrow \Q \P \oplus \P \Q^{(2)} \P^{(2)} \Q \rightarrow \P \Q^{(2)} \P \{1\} \right] \\
\label{calc:4}&\cong \left[ \P \Q \{0,1\} \rightarrow \P \Q \oplus \id \{0,1\} \oplus \P \Q \{0,1\} \rightarrow \P \Q \{0,1\} \right].
\end{align}

Claim: the right hand map in (\ref{calc:4}) has rank two, allowing us to cancel out four copies of $\P \Q$. In light of non-example 3 this claim is not so easily believable. Nevertheless, a careful computation shows that the matrix representing this map is diagonal, with isomorphisms on the diagonal, and hence has rank two. 

This leaves us with $[\P \Q \{0,1\} \rightarrow \P \Q \oplus \id \{0,1\}]$. We can cancel one more pair of $\P \Q$ on the left to leave us with $\P \Q \{1\} \rightarrow \id \{0,1\}$. Thus 
\begin{equation}\label{calc:6}
\Cone(\id \rightarrow \Phi_{-1}^* \Phi_{-1}) \cong [\P \Q \{1\} \rightarrow \id \{1\}].
\end{equation}
On the other hand, $\Phi_{-1} \Phi_{-1}^* \cong \P \Q$. Thus $\Cone(\Phi_{-1} \Phi_{-1}^* \rightarrow \id) \cong [\P \Q \rightarrow \id]$. Comparing this with (\ref{calc:6}) gives us relation (\ref{R4t}) from Conjecture \ref{conj:2}. 

\subsubsection{Example 5}\label{sec:ex5} Assume that $ \Q^n = 0$ for $n \geq 1$ so that
\begin{align}
\nonumber \Phi_{-1} \Phi_{-1}
&\cong [\P^{(2)} \Q \{1\} \rightarrow \P][\P] \\
\label{calc:7} &\cong \left[\P^{(2)} \{1\} \oplus \P^{(2)} \rightarrow \P \P \right].
\end{align}
The map in (\ref{calc:7}) is $\begin{pmatrix} C_2 \iota & -\iota \end{pmatrix}$ where $ \iota $ is the inclusion $ \P^{(2)} \rightarrow \P\P $ and $ C_1 $ and $ C_2$ are hollow dots on the left and right upward pointing strands respectively. 
This map is invertible, with inverse $\P \P \rightarrow \P^{(2)} \{1\} \oplus \P^{(2)}$ given by
\begin{equation*}
\begin{pmatrix}
\pi C_2 - \pi C_1 \\
\pi C_1 C_2 - \pi
\end{pmatrix}
\end{equation*}
where $ \pi \colon \P \P \rightarrow \P^{(2)} $ is the projection map. Thus $ \Phi_{-1} \Phi_{-1} \cong 0$.

\subsubsection{Example 6}\label{sec:ex6} Suppose $\Q^n=0$ for $n>1$. Then 
\begin{align}
\Phi_0 \Phi_0
\nonumber &\cong [\P \Q \{1\} \rightarrow \id][\P \Q \{1\} \rightarrow \id] \\
\nonumber &\cong [\P \Q \P \Q \rightarrow \P \Q \{1,1\} \rightarrow \id] \\
\label{calc:5} &\cong [\P \Q \{0,1\} \rightarrow \P \Q \{1,1\} \rightarrow \id].
\end{align}
It is not hard to check that one can cancel out two copies of $\P \Q \{1\}$ above, leaving us with $\P \Q \rightarrow \P \Q \{1\} \rightarrow \id$. However, at this point the first map cannot be an isomorphism since, if it were, the composition with the second map would be nonzero (contradicting the fact that (\ref{calc:5}) is a complex). Thus $\Phi_0^2$ is not isomorphic to $\id$ (although it is equal to $\id$ at the level of Grothendieck groups). 

As this example illustrates, when applying the cancellation Lemma \ref{lem:cancel} one has to be careful. The first map in (\ref{calc:5}) is an isomorphism between any choices of summands on the left and the right. However, after applying Lemma \ref{lem:cancel}, the remaining map becomes zero ({\it i.e.} the map $A-BD^{-1}C$ from Lemma \ref{lem:cancel} is zero). An equivalent way of saying this is that the rank of the left map in (\ref{calc:5}) is one (rather than two). This can be checked by direct computation rather than the indirect argument given above.

\section{Further questions and remarks}\label{sec:questions}

We finish with some general remarks and open questions. 

\subsection{A categorical version of the Clifford algebra}

Let us restrict our attention to the untwisted case (one can also pose the same questions about the twisted case). The discussion and conjectures above ignore what is perhaps the most interesting aspect of this story, the natural transformations between the $\Psi$'s. More precisely, the Heisenberg category $\H$ is equipped with a collection of morphisms described diagrammatically in section \ref{sec:H}. What is the analogue of these maps for the Clifford algebra?

An answer to this question involves a category $\Cl$ with objects (or 1-morphisms) generated by $\Psi_i$'s and a collection of morphisms (or 2-morphisms) between them. This collection of maps should induce the relations described in Conjecture \ref{conj:1}. In fact, $\Cl$ should have the structure of a dg-category. For instance, as observed in \cite{Kh2} in a different context, the relation $\Psi_i^2 \cong 0$ should be encoded by the existence of a map $T_{ii}: \Psi_i \Psi_i \rightarrow \Psi_i \Psi_i$ whose differential is equal to the identity.

With the right definition of $\Cl$, Conjecture \ref{conj:1} should then be lifted to a functor $\Cl \rightarrow \H$ where $\Psi_i$ and $\Psi_i^*$ are mapped to the complexes as in Conjecture \ref{conj:1}. In principle, it should be possible to guess the 2-morphisms in $\Cl$ by studying the maps between the various complexes (\ref{eq:cpxC1}) and (\ref{eq:cpxC2}) in $\Kom(\H)$. For instance, one should be able to actually see a map $T_{ii}$ mentioned above. 

Unfortunately, this is hard since it is difficult to compute all the maps between various compositions of complexes such as those (\ref{eq:cpxC1}) and (\ref{eq:cpxC2}). Alternatively, one may try to guess the right definition of $\Cl$ and then justify this guess by defining the functor $\Cl \rightarrow \H$. In the literature there are some categorifications of Clifford algebras, such as that discussed in \cite{Ti}. Unfortunately, the presentation of the Clifford algebra used in \cite{Ti} is not the one in the current paper and we do not currently understand a sensible relationship between the categorification in \cite{Ti} and the hoped for dg-category $\Cl$. 

Finally, one should note that the Boson-Fermion correspondence also recovers the action of the Heisenberg from that of the Clifford. We have not addressed this side of the correspondence because to do this we would need to know the whole structure of $\Cl$. This part of the correspondence should give a functor $\H \rightarrow \Cl$. 

\subsection{Semi-simplicity and the Fock space}

The categories $\k[S_n]\mod$ used to define $\V_{Fock}$ (in the untwisted case) are semi-simple. In some ways this is unfortunate because it means that the complexes (\ref{eq:cpxC1}) and (\ref{eq:cpxC2}) used to define $\Psi_i$ and $\Psi_i^*$ must split ({\it i.e.} they are not really complexes and hence less interesting). We address this issue with two remarks. 

Firstly, even if these complexes split it is still interesting to wonder what $(\k[S_m],\k[S_n])$-bimodule they correspond to (in some sense this is done in \cite{J} in terms of characters). Whatever the answer may be it is noteworthy that it can be expressed naturally in terms of these categorical vertex operators.

Secondly, there are other versions of the categorical Fock space $\V_{Fock}$ which may not be semi-simple. For instance, instead of $\V_{Fock}$ one may take $\H$ quotiented by the ideal generated by objects 
of the form $\X \Q$ where $\X$ is an arbitrary object of $\H$. Note that the objects here are products of $\P$'s (just like in $\V_{Fock}$) but this category is actually larger than $\V_{Fock}$ since it contains more morphisms. This category is actually the universal categorical Fock space in the sense that it maps to any other categorification of Fock space. In this case our conjectures still apply but the complexes such as in (\ref{eq:cpxC1}) and (\ref{eq:cpxC2}), which define $\Psi_i$ and $\Psi_i^*$, do not necessarily split (as far as we know). 

\subsection{The affine Sergeev algebra}

Inside $\H$ or $\tH$ consider the element $X_i \in \End(\P^n)$, acting on the $i$th factor $\P$, as illustrated in the left hand side of (\ref{eq:curl}). In \cite{Kh1} this element was studied and it was encoded diagrammatically by a solid dot, as shown on the right side of (\ref{eq:curl}). 
\begin{equation}\label{eq:curl}
\begin{tikzpicture}[>=stealth]
\draw  (1.9,0) .. controls (1.9,.5) and (1.3,.5) .. (1.1,0) [thick];
\draw  (1.9,0) .. controls (1.9,-.5) and (1.3,-.5) .. (1.1,0) [thick];
\draw  (1,-1) .. controls (1,-.5) .. (1.1,0) [thick];
\draw  (1.1,0) .. controls (1,.5) .. (1,1) [->] [thick];
\draw  (2.5,0) node {$=$};
\draw (3,-1) -- (3,1) [thick][->];
\filldraw [black] (3,0) circle (2pt);
\end{tikzpicture}
\end{equation}
In \cite{Kh1} it was shown that these $X_i$'s together with the symmetric group $\k[S_n] \subset \End(\P^n)$ generate a copy of the degenerate affine Hecke algebra. In the twisted case, it turns out that these elements, together with $\Se_n \subset \End(\P^n)$ generate a copy of a version of the affine Sergeev algebra (which was introduced in \cite{Ser}).   
More precisely, denote a crossing of the $i$th and $(i+1)$st strands by $T_i$ and let a hollow dot on the $i$th strand be denoted by $C_i$.  
\begin{prop}\label{prop:sergeevacting}
For the twisted Heisenberg $\tH$, the following relations hold inside $\End(\P^n)$: 
\begin{align*}
T_i X_i &= X_{i+1}T_i + 1 + C_iC_{i+1} \\
X_i T_i &= T_i X_{i+1} + 1 - C_i C_{i+1} \\
C_i X_j &= (-1)^{\delta_{i,j}} X_j C_i \\
X_i X_j &= X_j X_i.
\end{align*}
\end{prop}
\begin{proof}
These are proved by direct diagrammatic calculations (as in \cite{Kh1} for the case of $\H$). 
\end{proof}
Iterating these relations one also finds that 
\begin{align*}
T_i X_i^k &= X_{i+1}^kT_i + \sum_{j=0}^{k-1} X_i^j X_{i+1}^{k-1-j}  + \sum_{j=0}^{k-1} (-1)^{j} X_i^j X_{i+1}^{k-1-j} C_iC_{i+1} \\
X_i^k T_i &= T_i X_{i+1}^k + \sum_{j=0}^{k-1} X_i^j X_{i+1}^{k-1-j}  + \sum_{j=0}^{k-1} (-1)^{k-j} X_i^j X_{i+1}^{k-1-j} C_iC_{i+1}.
\end{align*}
This version of the affine Sergeev algebra acting on $\End(\P^n) $ is related by an anti-isomorphism to the algebra considered in ~\cite{HKS}.

\end{document}